\documentclass[12pt,leqno,a4paper]{amsart}
\usepackage{amsmath,amsthm,amsfonts,amssymb,latexsym, mathabx}
\usepackage{enumerate}
\usepackage{color, xcolor}

\newcommand{\bG}{{\mathbf{G}}}
\newcommand\bg[1]{\mathbf{#1}}

\newcommand{\out}{{\operatorname{Out}}}

\newcommand{\PGL}{{\operatorname{PGL}}}
\newcommand{\SL}{{\operatorname{SL}}}
\newcommand{\Sp}{{\operatorname{Sp}}}
\newcommand{\SO}{{\operatorname{SO}}}
\newcommand{\GO}{{\operatorname{GO}}}

\newcommand{\PSL}{{\operatorname{PSL}}}
\newcommand{\PSU}{{\operatorname{PSU}}}

\newcommand{\Syl}{{\operatorname{Syl}}}

\newcommand{\tw}[1]{{}^#1\!}

\def\syl#1#2{{\rm Syl}_#1(#2)}
\def\nor{\triangleleft\,}

\def\oh#1#2{{\bf O}_{#1}(#2)}
\def\Oh#1#2{{\bf O}^{#1}(#2)}
\def\zent#1{{\bf Z}(#1)}
\def\irr#1{{\rm Irr}(#1)}
\def\norm#1#2{{\bf N}_{#1}(#2)}
\def\cent#1#2{{\bf C}_{#1}(#2)}

\def\irrp#1#2{{\rm Irr}_{#1'}(#2)}

\def\sbs{\subseteq}

\newtheorem{thm}{Theorem}[section]
\newtheorem{lem}[thm]{Lemma}

\newtheorem{prop}[thm]{Proposition}

\newtheorem*{thm*}{Theorem $\heartsuit$}

\newtheorem{thml}{Theorem}

\theoremstyle{remark}

\newcommand\wt[1]{\widetilde{#1}}

\newcommand{\GL}{\operatorname{GL}}

\newcommand{\PSp}{\operatorname{PSp}}
\newcommand\type[1]{\operatorname{#1}}
\newcommand\cyc[1]{\sf{#1}}

\newcommand{\aut}{\operatorname{Aut}}

\usepackage{hyperref}
\begin{document}

\title[Principal blocks with 5 irreducible characters]{Principal blocks with 5 irreducible characters}


\author{Noelia Rizo}
\address[N. Rizo]{Departamento de Matem\'aticas, Universidad del País Vasco UPV/EHU}
\email{noelia.rizo@ehu.eus}

\author{A. A. Schaeffer Fry}
\address[A. A. Schaeffer Fry]{Dept. Mathematics and Statistics, MSU Denver, Denver, CO 80217, USA}
\email{aschaef6@msudenver.edu}

\author{Carolina Vallejo}
\address[C. Vallejo]{Departamento de Matem\'aticas, Edificio Sabatini, Universidad Carlos III de Madrid,
Av. Universidad 30, 28911, Legan\'es. Madrid, Spain}
\email{carolina.vallejo@uc3m.es}

\thanks{The first and third authors are  partially supported by the
Spanish Ministerio de Ciencia
e Innovaci\'on
PID2019-103854GB-I00 and FEDER funds. The first author also acknowledges support by ``Convocatoria de contratación para la especialización de personal investigador doctor en la UPV/EHU (2019)''. The second author is partially supported by the National Science Foundation under Grant No. DMS-1801156.
The third author also acknowledges support by the Spanish Ministerio de Ciencia
e Innovaci\'on through SEV-2015-0554
and 
MTM2017-82690-P}



\begin{abstract} We show that if the principal $p$-block of a finite group $G$ contains exactly 5 irreducible ordinary characters, then a Sylow $p$-subgroup
of $G$ has order 5, 7 or is isomorphic to one of the non-abelian 2-groups of order 8. 
 \end{abstract}

\maketitle

\section*{Introduction}

\noindent 
The problem of classifying finite groups $G$ depending on the number $k(G)$ of irreducible ordinary characters of $G$ goes back to Burnside \cite{Bur55}, and a complete classification up to $k(G)=14$ has been achieved in \cite{SV07}. On the other hand, in Modular Representation Theory, very little is known about the analogous number $k(B)$ of irreducible ordinary characters belonging to a Brauer $p$-block $B$ of $G$. 
Brauer's celebrated $k(B)$-conjecture asserts that the number of irreducible characters in a $p$-block $B$ is bounded from above by the order of a defect group $D$ of $B$. This conjecture remains open (and even without a reduction theorem to finite simple groups) at the present time. In this paper, we are concerned with the related problem of classifying the defect groups of $p$-blocks $B$ with a small number $k(B)$, which already appears to be difficult. 

It is well-known that $k(B)=1$ if, and only if, $D$ is trivial. It is also true that $k(B)=2$ if, and only if, $D\cong {\sf C}_2$ (\cite{Bra}).
K\"ulshammer, Navarro, Sambale, and Tiep conjecture that $k(B)=3$ if, and only if, $D\cong {\sf C}_3$ in \cite{KNST}. They prove that their conjecture holds for a $p$-block $B$ if
that block satisfies the statement of the  
Alperin-McKay conjecture in \cite[Theorem 4.2]{KNST}. The case where $l(B)=1$, that is, the $p$-block $B$ has a unique simple module, was proved in \cite{Kul84}. In general, this conjecture is still open.

More can be said if we restrict ourselves to the case of principal $p$-blocks. In the following, we denote by $B_0(G)$ (or simply $B_0$ if the context is clear) the principal $p$-block of $G$, so that $D \in \syl p G$.
For example, Belonogov showed that $k(B_0)=3$ if, and only if, $D\cong {\sf C}_3$ in \cite{Bel90}. This paper is not easily available. Recently, Koshitani and Sakurai have provided in \cite{KS20} an alternative proof of this result and we will rely on their proof throughout this article.
Moreover, Koshitani and Sakurai \cite{KS20} have shown
that $k(B_0)=4$ implies that $D \in \{  {\sf C}_2\times {\sf C}_2, {\sf C}_4, {\sf C}_5\}$. In this note, we go one step further and
analyze the isomorphism classes of Sylow $p$-subgroups of groups $G$ for which $B_0(G)$ has exactly 5 irreducible characters. 

\begin{thml}\label{thmA} Let $G$ be a finite group, $p$ be a prime, and $P\in \syl p G$. Denote by $B_0$ the principal $p$-block of $G$. If $k(B_0)=5$,  then $P\in \{ {\sf C}_5, {\sf C}_7, {\sf D}_8, {\sf Q}_8 \}$. 
\end{thml}

Notice that all the groups described in of Theorem \ref{thmA} occur as Sylow $p$-subgroups of principal $p$-blocks with 5 irreducible characters. For example, $G={\sf C}_5$ and $p=5$ yields the case $P={\sf C}_5$, $G={\sf D}_{14}$ and $p=7$ yields the case $P={\sf C}_7$,  $G={\sf D}_{8}$ and $p=2$ yields the case $P={\sf D}_8$, and  $G={\sf Q}_{8}$ and $p=2$ yields the case $P={\sf Q}_8$. 

\smallskip

We remark that for general $p$-blocks $B$ with $k(B)=5$ and $l(B)=1$,  Chlebowitz and K\"ulshammer proved that $D \in \{ {\sf C}_5, {\sf D}_8, {\sf Q}_8\}$ in \cite{CB92}.

\smallskip

Our proof of Theorem \ref{thmA} is based on the analysis of the structure of a counterexample $G$ of minimal order. We first prove that such a $G$ must be almost simple in Theorem \ref{thmB}.
Our methods for proving Theorem \ref{thmB} can also be used to provide independent reductions to simple and, respectively, almost simple groups of the corresponding statements for $k(B_0)=3$ \cite[Theorem 3.1]{KS20} and $k(B_0)=4$ \cite[Theorem 5.1 and 6.1]{KS20}.  As explained by the authors in \cite{KS20}, their reductions depend on the study of the structure of Sylow $p$-subgroups of finite groups having exactly 2 $G$-conjugacy classes of $p$-elements carried out in \cite{KNST}. The main difference is that our methods would remove the dependence on that result. 

With Theorem \ref{thmB} in place and by using the Classification of Finite Simple Groups, we are left to study the principal $p$-blocks of finite almost simple groups whose socle is a 
group of Lie type.  This is done in Section \ref{sec:lie}. 


\medskip

We finish this introduction by mentioned that, as noticed by one of the referees, the statement of Theorem A follows from the principal block case of the Alperin-McKay conjecture. In particular, we can see Theorem A as new evidence in support of this long-standing conjecture.

\section{Preliminaries and background results}

\noindent Our notation for characters and $p$-blocks follows \cite{nbook}. For a fixed prime $p$, we refer to $p$-blocks just as blocks. In this section we collect some results on blocks that we will use later on, as well as
background results on blocks with a small number of irreducible characters.

We start by recalling some facts about {\em covering blocks}. For a definition and first properties of covering blocks we refer the reader to \cite[Chapter 9]{nbook}. Recall that if $N$ is a normal subgroup of $G$ and $B$ and $b$ are blocks of $G$ and $N$ respectively, then $B$ covers $b$ if there are $\chi\in{\rm Irr}(B)$ and $\theta\in{\rm Irr}(b)$ such that $\theta$ is an irreducible constituent of the restriction $\chi_N$ (see \cite[Theorem 9.2]{nbook}). Then it is clear that $B_0(G)$ covers $B_0(N)$.

\begin{lem}\label{lem:blockabove}
Let $N \nor G$.  Then for every $\theta\in\irr{B_0(N)}$, there exists $\chi\in\irr{B_0(G)\mid \theta}$.
\end{lem}
\begin{proof}
This is \cite[Theorem 9.4]{nbook}.
\end{proof}

We will often use the following two facts in Sections \ref{sec:reduction} and \ref{sec:lie}.

\begin{lem}\label{covering}
Let $N \nor G$. Suppose that $B\in {\rm Bl}(G)$ is the only block covering $b \in {\rm Bl}(N)$. Then for every $\theta \in \irr b$, the set $\irr{G| \theta}\sbs \irr{B}$. 
\end{lem}
\begin{proof}
Let $\chi\in{\rm Irr}(G|\theta)$ and let ${\rm Bl}(\chi)$ be the $p$-block of $G$ containing $\chi$. Then ${\rm Bl}(\chi)$ covers $b$ (see \cite[Theorem 9.2]{nbook}) and hence ${\rm Bl}(\chi)=B$.
\end{proof}

\begin{lem}\label{onlycovering}
Let $M \nor G$ and $P \in \syl p G$. If $P \cent G P \sbs M$, then $B_0(G)$ is the only block covering $B_0(M)$. In particular, $k(G/M)<k(B_0(G))$ as long as $P>1$.
\end{lem}
\begin{proof} Write $B_0=B_0(G)$ and $b_0=B_0(M)$ to denote the principal $p$-blocks of $G$ and $M$ respectively. 
By \cite[Theorem 4.14 and Problem 4.2]{nbook} and Brauer's third main theorem \cite[Theorem 6.7]{nbook}, we have that $b_0^G=B_0$. Let $B$ be another block of $G$ covering $b_0$. By \cite[Theorem 9.26]{nbook}, $B \in {\rm Bl}(G|P)$. By \cite[Theorem 9.19, Lemma 9.20]{nbook} we have that $B$ is regular with respect to $M$ and hence $B=b_0^G=B_0$. The last part follows from Lemma \ref{covering} as $k(b_0)>1$ if $P>1$.
\end{proof}
Next we record some observations that will be useful when dealing with quotients.  The following can be seen using Clifford theory.  A proof may be found, for example, in \cite[Lemma 3.2.7]{SFthesis}.

\begin{lem}\label{lem:restnumber}
Let $G$ be a finite group and $N\lhd G$ such that $G/N$ is cyclic.  Let $\chi\in\irr{G}$.  Then the number of irreducible constituents of the restriction $\chi_N$ is the number of $\beta\in\irr{G/N}$ such that $\chi\beta=\chi$.
\end{lem}

The next lemma can be found, for example, as \cite[Lemma 17.2]{CE04}.

\begin{lem}\label{lem:quotientblock}
Let $G$ be a finite group.  Two characters of $G/\zent{G}$ are in the same block if and only if they are in the same block as a character of $G$.
\end{lem}

If $B$ is a $p$-block with defect group $D$, then
 ${\rm Irr}_{0}(B)=\{ \chi \in {\rm Irr}(B) \ | \chi(1)_p=|G:D|_p \}$, where $n_p$ denotes the $p$-part of the integer $n$. We write $k_0(B)=|{\rm Irr}_{0}(B)|$ for the number
 of height zero characters in $B$ (if $D>1$ then it is well-known that $k_0(B)\geq 2$).
  If $B=B_0(G)$ is the principal $p$-block
 of $G$, then $D \in \syl p G$ and ${\rm Irr}_0(B)$ is the subset of $p'$-degree characters of $B$.

\smallskip

We will often use the following results on the number of height zero characters in (principal) blocks for $p\in \{2, 3\}$ in Section 2.

\begin{thm}\label{divisibility2} Let $B$ be a 2-block of $G$ with defect group $D$. 
\begin{enumerate}[{\rm (a)}]
\item If $2 \mid |D|$ then $2\mid k_0(B)$.
\item If $4 \mid |D|$ then $4\mid k_0(B)$.
\item $|D|=2$ if, and only if, $k_0(B)=2$. 
\end{enumerate}
\end{thm}
\begin{proof}
Parts (a) and (b) can be proven following the ideas in the proof of \cite[Lemma 2.2]{NST18} (part (b) is \cite[Corollary 1.3]{Lan81}).
The ``if'' implication in part (c) follows from parts (a) and (b). If $|D|=2$, then the block $B$ is nilpotent and by \cite[Theorem 1.2, Corollary 1.4]{BP80} $k_0(B)=2$. 
 \end{proof}
 
  
 \medskip 

We remark that, although Theorem \ref{divisibility2} does not rely on the Classification of Finite Simple Groups, the following extension to the case $p=3$ (more specifically, part (b) in Theorem \ref{divisibility3} below) does.

\begin{thm}\label{divisibility3} Let $B$ be a principal 3-block of $G$ with defect group $D$. 
\begin{enumerate}[{\rm (a)}]
\item If $3\mid |D|$ then $3\mid k_0(B)$.
\item If $B=B_0(G)$ is the principal 3-block of $G$, then $|D|=3$ if, and only if, $k_0(B)=3$ 
\end{enumerate}
\end{thm}
\begin{proof}
Part (a) is \cite[Corollary 1.6]{Lan81}. 
Part (b) is \cite[Theorem C]{NST18}. 
\end{proof}

We summarize below what is known on (principal) blocks with up to 4 irreducible characters. Our proof of Theorem \ref{thmA}
relies on those facts.

\begin{thm}\label{smaller4}
Let $B$ be a $p$-block of a finite group $G$ with defect group $D$. 
\begin{enumerate}[{\rm (a)}]
\item $k(B)=1$ if, and only if, $D=1$.
\item $k(B)=2$ if, and only if, $D\cong {\sf C}_2$.
\item If $B=B_0(G)$ the principal $p$-block, then  $k(B)=3$ if, and only if, $D \cong {\sf C}_3$.
\item If $B=B_0(G)$ the principal $p$-block, then $k(B)=4$ implies $D \in \{  {\sf C}_2\times {\sf C}_2, {\sf C}_4, {\sf C}_5 \}$. 
\end{enumerate}
\end{thm}
\begin{proof}
Part (a) is \cite[Theorem 3.18]{nbook}. 
Part (b) is \cite[Theorem A]{Bra}.
The ``only if'' implication in  part (c) follows from \cite{Kul84} (when $l(B)=1$) and \cite{KS20} (when $l(B)=2$). The converse in part (c) follows since $k(B)\leq 3$ by \cite[Theorem 1*]{BF59} and parts (a) and (b).
Part (d) is \cite[Theorem 1.1]{KS20} (note that if $l( B_0(G))=1$, then $G$ has a normal $p$-complement $K$ and $k( B_0(G))=k(B_0(G/K))=k(D)=4$).
\end{proof}

\smallskip

\section{A reduction to almost simple groups}\label{sec:reduction}

\noindent The aim of this section is to prove the following reduction theorem. 

\begin{thm}\label{thmB} Let $G$ be a finite group, $p$ a prime and $P\in \syl p G$. If $G$ is a minimal counterexample to the statement of Theorem A, then $G$ is almost simple. Write $S={\rm soc}(G)$ and $\bar B_0=B_0(G/S)$. Then
either $p$ does not divide $|G:S|$ and $S\cent G P <G$, or $k(\bar B_0)=4$.
Moreover $S$ is a simple group of Lie type not isomorphic to an alternating or sporadic group.
\end{thm}

The following technical result will be useful for handling some cases in the proof of Theorem \ref{thmB}.

\begin{lem}\label{tech} Let $M \nor G$ be such that $P\cent G P \subseteq M$, where $P\in \syl p G$. Write $B_0=B_0(G)$ and $b_0=B_0(M)$. If $k(B_0)=5$, then $k(b_0)\in \{ 4, 5, 7, 11, 13\}$. Moreover, if $p=2$ and $M<G$, then $k(b_0)=7$.
\end{lem}
\begin{proof}  By Lemma \ref{onlycovering}, the block
$B_0$ is the only block covering $b_0$, so that  $\irr{G/M}\sbs \irr{B_0}$ and $1\leq k(G/M)<5$.
If $G/M=1$, then $k(b_0)=k(B_0)=5$, and we are done. Hence 
 $1<k(G/M)<5$. By \cite{VV85}, $G/M \in \{ {\sf C}_2, {\sf C}_3, {\sf S}_3, {\sf C}_4, {\sf C}_2\times {\sf C}_2, {\sf D}_{10}, {\sf A}_4\}$.
We analyze $k(b_0)$ depending on the isomorphism class of $G/M$.

Suppose that $G/M\cong {\sf C}_2$. Write $\irr{B_0}=\{ 1_G, \lambda, \chi, \psi, \xi \}$, where $\lambda$ lies over $1_M$. If some $1_M\neq \theta \in \irr{b_0}$ is $G$-invariant, then $\theta$ has two extensions, say $\chi$ and $\psi$, both in $\irr{B_0}$ by Lemma \ref{covering}.  Then $\xi_M$ decomposes as a sum of two distinct elements of $\irr{B_0}$ and  $k(b_0)=4$. Otherwise, every character in $\irr{B_0}\setminus \irr{G/M}$ is induced from $M$ and $k(b_0)=7$. 

Suppose  that $G/M\cong {\sf C}_3$, reasoning  as before one can show  that every character in $\irr {B_0}\setminus \irr{G/M}$ must be induced from an irreducible character of $M$, hence $k(b_0)=7$. 

Suppose that $G/M \cong {\sf S}_3$. Let $M\leq K\nor G $ with $|K:M|=3$. Note that $B_0(K)$ is the only block covering $b_0$ and $B_0$ is the only block covering $B_0(K)$ by Lemma \ref{onlycovering}. If some $1_M\neq \theta \in \irr{b_0}$ is $G$-invariant, then $\theta$ has 3 extensions in $B_0(K)$ and at least one of them, namely $\xi$, is $G$-invariant by counting. Hence $\xi$ has two extensions in $B_0$. That leaves no space for characters over the other two extensions of $\theta$ to $K$. Hence, we may assume no nontrivial character in $\irr{b_0}$ is $G$-invariant. Similarly, if some $1_M\neq \theta \in \irr{b_0}$ is $K$-invariant, using the Clifford correspondence, we would obtain 3 distinct characters in $\irr{B_0}$ not containing $M$ in their kernels, absurd. 
Write $\irr{B_0}=\{ 1_G, \lambda, \mu, \chi, \psi\}$, where $\lambda$ and $\mu$ lie over $1_M$. Let $\theta \in \irr {b_0}$ lie under $\chi$ or $\psi$. If $M<G_\theta$, then $|G_\theta: M|=2$, then $\theta$ has 2 extensions  $\xi_i \in \irr{B_0(G_\theta)}$ for $i=1,2$ and, since $B_0$ is the only block covering $b_0$, by the Clifford correspondence  $\xi_i^G\in \irr {B_0} $ for $i=1,2$ are distinct so $\{ \xi_i^G\}=\{ \chi, \psi\}$. In particular, $k(b_0)=2$. In this case $P \cong {\sf C}_2$ which implies $k(B_0)=2$, a contradiction. Hence  $G_\theta= M$, so $\chi_M$ and $\psi_M$ decompose as a sum of $|G/M|$ distinct characters, yielding $k(b_0)=13$.

Otherwise $k(G/M)=4$, and consequently $G/M$ acts transitively on the set of nontrivial characters of $\irr{b_0}$, all of which lie under the only $\chi \in \irr{B_0}$ that does not lie over $1_M$. In particular, by Clifford's theorem 
$k(b_0)=1+|G:G_\theta|$, where $1_M\neq \theta \in \irr{b_0}$. 
If $G_\theta=G$, then $k(b_0)=2$ and $P\cong {\sf C}_2$ contradicting $k(B_0)=5$. 
Reasoning as before, one can check case by case that if $1_M\neq \theta \in \irr{b_0}$, then $G_\theta=M$ or $|G_\theta:M|=4$.
Hence \cite{VV85} implies that
$k(b_0)=1+|G:G_\theta|\in \{4, 5, 11, 13\}$ completing the proof. Note that if $p=2$ then $|G/M|$ is odd and hence $G/M\cong {\sf C}_3$ and $k(b_0)=7$.
\end{proof}

Our proof of Theorem \ref{thmB} uses the principal block case of a celebrated result of Kessar--Malle, whose proof relies on the Classification of Finite Simple Groups. 

\begin{thm}\label{heightzero}\cite{KM13} Let $B$ be the principal block of $G$ and let $P\in{\rm Syl}_{p}(G)$. 
If $P$ is abelian then $k(B)=k_0(B)$. 
\end{thm}

We will also make use of Alperin-Dade's theory of isomorphic blocks.

\begin{thm}\label{isomblocks}
Suppose that $N$ is a normal subgroup of $G$, with $G/N$ a $p'$-group.
Let $P \in \syl p G$ and assume that $G=N\cent GP$. Then restriction of characters defines
a natural bijection between the irreducible characters of the principals blocks
of $G$ and $N$. 
\end{thm}

\begin{proof}
The case where $G/N$ is solvable was proved in \cite{Alp76} and the general case
in \cite{Dad77}. 
\end{proof}


\begin{proof}[Proof of Theorem \ref{thmB}] 
If $x$ is a $p$-element of $G$, then $B_0(\cent{G}{x})^G$ is defined by \cite[Theorem 4.14]{nbook} and hence, by  \cite[Theorem 5.12]{nbook} and Brauer's third main theorem \cite[Theorem 6.7]{nbook}, we have that 

\begin{equation}\label{eq:k(B)}5=k(B_0)=|\zent G|_p l(B_0)+\sum_{i=1}^k l(B_0({\rm \textbf{C}}_G(x_i))\, ,
\end{equation}
 where $\{x_1,x_2,\ldots,x_k\}$ is a complete set of representatives of the conjugacy classes of non-central $p$-elements of $G$.

{\em Step 1.} \textit{We may assume that $\oh {p'} G=1$}. This is due to the minimality of $G$ as a counterexample and  \cite[Theorem 9.9.(c)]{nbook}.

\medskip

{\em Step 2.} \textit{We may assume that $G$ is not $p$-solvable}. Otherwise, by Step 1 and \cite[Theorem 10.20]{nbook}, $G$ has a unique $p$-block and hence $k(B_0)=k(G)=5$. In this case, by \cite{VV85} we have that $G\in \{ {\sf C}_5, {\sf D}_8, {\sf Q}_8, {\sf D}_{14}, {\sf C }_5\rtimes {\sf C}_4, {\sf C}_7 \rtimes {\sf C}_3, {\sf S}_4\}$
(the prime is clear in each case since $G$ has a unique $p$-block). Then $P\in \{ {\sf C}_5, {\sf C}_7, {\sf D}_8, {\sf Q}_8 \}$, and $G$ is not a counterexample.

\medskip

{\em Step 3.} \textit{We have $1<l(B_0)<5$ and $1\leq k \leq 3$.} If $l(B_0)=1$ or $P\sbs\zent G$ then $G$ has a normal $p$-complement (by \cite[Corollary 6.13]{nbook} or by the Schur-Zassenhaus theorem respectively), contradicting Step 2.

\medskip

{\em Step 4.}  \textit{We may assume that $\zent G=1$.} Indeed, by Step 1, $\zent G$ is a $p$-group. If $\zent G>1$, Equation \ref{eq:k(B)} and Step 3 force $|\zent G|= 2$, $l(B_0)=2$, $k=1$ and $l(B_0(\cent G {x_1}))=1$. Note that $p=2$ in this case. Write $N=\zent G$ and $\bar B_0=B_0(G/N)\subseteq B_0$. Since $N$ is a nontrivial 2-group, we have that $k(\bar B_0)<k(B_0)$. By  \cite[Theorem 9.10]{nbook} we have $l(\overline B_0)=l(B_0)=2$. 

If $x_1 N \in \zent{G/N}$, since $G$ has a unique conjugacy class of non-central 2-elements, then $P/N\sbs \zent {G/N}$. In particular, $P\nor G$ contradicting Step 2. Otherwise $x_1N$ is a non-central $p$-element of $G/N$ and then, applying \cite[Theorem 5.12]{nbook} to $G/N$, we have $2\leq |\zent {G/N}|_2l(\bar B_0)< k(\bar B_0)\leq 4$. Then $k(\bar B_0)\in \{ 3, 4\}$. If $k(\bar B_0)=3$, we have that $p=3$, a contradiction. If $k(\bar B_0)=4$, since $p=2$ we get $|P/N|=4$ by Theorem \ref{smaller4}(d). Hence $|P|=8$ and is non-abelian (if $P$ is abelian we get 4 divides 5 by Theorem \ref{divisibility2} and Theorem \ref{heightzero}), and $G$ is not a counterexample.







\medskip

{\em Step 5.} $G$ has a unique minimal normal subgroup $N$.

By Step 1, $p$ divides the order of every minimal normal subgroup of $G$. Since $k\leq 3$ and $\zent G =1$, if $G$ has
more than one minimal normal subgroup, then it has exactly two, namely $N_1$ and $N_2$, and $k=3$. Write $N=N_1\times N_2\nor G$. 
Note that $G$ has no $p$-elements outside $N$, so $P\subseteq N$ and $G/N$ is a $p'$-group. Let $1\neq x_1 \in N_1$ and $1\neq x_2 \in N_2$, be $p$-elements.  By Equation \ref{eq:k(B)} we know that $l(B_0(\cent G{x_i}))=1$ hence the groups $\cent G{x_i}$ have normal $p$-complements by \cite[Corollary 6.13]{nbook}. Since $N_2 \leq \cent G {x_1}$ and $N_1 \leq \cent G{x_2}$, we see that $N_i$ have normal $p$-complements too. By Step 1 this forces $N_i$ to be $p$-groups. Then $P=N\nor G$ and $G$ is $p$-solvable, contradicting Step 2.


\medskip

Write $\bar{B_0}=B_0(G/N)$. Since $\bar B_0\subseteq B_0$ 
we have that $1\leq k(\bar B_0)\leq 5$. Moreover $k(\bar B_0)<5$ as otherwise $k(B_0(N))=1$ (contradicting Step 1). We have seen that
$1\leq k(\bar B_0)<5$.

\medskip

{\em Step 6.} $N$ is semisimple of order divisible by $p$.

Otherwise $N$ is abelian and by Step 1 we have that $N$ is an elementary abelian $p$-group. Note that $P/N>1$ as otherwise $G$ would be $p$-solvable contradicting Step 2. Hence, in this case, $1<k(\bar B_0)<5$.

Also by Step 1, we have that $\oh{p'} {\cent G N}=1$, and by Step 4, we have that $\cent G N <G$. Suppose that $\cent P N=\cent G N \cap P=N$. Then $N \in \syl p {\cent G N}$, and $\cent G N$ has a normal $p$-complement, forcing $\cent G N =N$. By \cite[Theorem 9.21]{nbook} $B_0$ is the only block covering $B_0(N)$. In particular, we have that $k(G/N)<5$. By the classification in \cite{VV85},  $G/N$ is solvable, so is $G$, a contradiction with Step 2. 
Hence $N<\cent P N$. We study now the different values of $k(\bar B_0)$.

\smallskip

 If $k(\bar{B}_0)=2$, then $p=2$ and $P/N\cong {\sf C}_2$ by Theorem \ref{smaller4}(b). It is well-known that a group with a cyclic Sylow 2-subgroup has a normal 2-complement (see Corollary 5.14 of \cite{isbook}, for instance). Then $G/N$ has a normal $p$-complement and therefore $G$ is $p$-solvable contradicting Step 2.
 
\medskip

If $k(\bar{B}_0)=3$, then $p=3$ and  $P/N\cong {\sf C}_3$ by Theorem \ref{smaller4}(c). Then  $P=\cent P N$, so that $N\subseteq \zent P$ and $P$ is abelian. By Theorem \ref{heightzero}, $k_0(B_0)=k(B_0)=5$. By Theorem \ref{divisibility3}, 3 divides $k_0(B_0))=5$, a contradiction.

\medskip

If $k(\bar{B}_0)=4$ then, by Theorem \ref{smaller4}(d), either $p=2$ and  $|P/N|=4$ 
or $p=5$ and $|P/N|=5$. Write $M=\cent G N$. Recall that $N<M\cap P\leq P$.

\smallskip

\noindent{$\bullet$} Suppose that $|P/N|=4$ (so $p=2$).

If  $|M\cap P:N|=2$, then $|G/M|_2=|P:M\cap P|_2=|M:N|_2=2$. Hence $G/M$ and $M/N$ have normal $p$-complements, and $G$ is $p$-solvable contradicting Step 2. 

It remains to consider the situation where $P\subseteq M$. In this case $P\cent G P \subseteq M$. 
By Lemma \ref{tech},  $k(B_0(M))=7\geq|\zent M|_pl(B_0(M))$ (using Theorem 5.12 of \cite{nbook}). In particular, since $N\sbs\zent M$, we have that $|N|$ divides 4. If $|N|=2$, then $|P|=8$ contradicting the minimality of $G$ as a counterexample (again $P$ is nonabelian as otherwise Theorem \ref{heightzero} and Theorem \ref{divisibility2} imply 4 divides 5). Otherwise $|N|=4$ and $l(B_0(M))=1$. Then $M$ has a normal 2-complement. By Step 1 this forces $M=P\nor G$, a contradiction with Step 2.

\smallskip

\noindent{$\bullet$} Suppose that $|P/N|=5$ (so $p=5$). Then $P\subseteq \cent G N=M$, so that Lemma \ref{tech} applies to $M\nor G$. By Lemma \ref{tech} and Theorem 5.12 of \cite{nbook}
$13\geq k(B_0(M))\geq |\zent M|_p l(B_0(M))$. Since $N\subseteq \zent M$, these inequalities force $N\cong {\sf C}_5$. Then $G/M\leq \aut (N) \cong {\sf C}_4$.  From the proof Lemma \ref{tech} one can actually conclude that $5\leq k(B_0(M)) \in \{4, 5, 7\}$, so that $l(B_0(M))=1$. Yielding  $M=P \nor G$, a contradiction with Step 2.

\medskip

{\em Step 7.} $N$ is (nonabelian) simple of order divisible by $p$. Moreover either $G/N$ is a $p'$-group or $k(\bar B_0)=4$.

By Step 6 we have that $N$ is semisimple. Write  $N=S_1\times S_2\times\cdots\times S_t$, where $S_i\cong S$ and $S$ is a nonabelian simple with $p\mid |S|$. 

\smallskip

 If $|G/N|$ is not divisible by $p$, then $P\subseteq N$. Let $M=N\cent G P\nor G$ by the Frattini argument. By Theorem \ref{isomblocks}, we have that $k(B_0(M))=k(B_0(N))=k(B_0(S))^t$.  By Lemma \ref{tech}, the equality $k(B_0(M))=k(B_0(S))^t$ with $t>1$ yields a contradiction unless $t=2$ and $k(B_0(S))=2$, which is absurd as $S$ is nonabelian simple.

\medskip

If $|G/N|$ divisible by $p$, then $1<k(\bar B_0)< 5$. We again study the different values of $k(\bar B_0)$.

Suppose that $k(\bar B_0)=2$. Then $p=2$ and  $|PN/N|=2$ by Theorem \ref{smaller4}(b). In particular, $G/N$ has a normal 2-complement $X/N$. By  \cite[Corollary 9.6]{nbook}, $B_0$ is the only block covering $B_0(X)$ and hence ${\rm Irr}(G/X)\sbs{\rm Irr}(B_0)$. Let $1_X\neq \tau\in{\rm Irr}(B_0(X))$. If $\tau$ extends to $G$ then it has two extensions, all in $B_0$ by Lemma \ref{covering}. Therefore $k(B_0(X))=4$. Since $p=2$ and $P\cap N=P\cap X\in{\rm Syl}_p(X)$, by Theorem \ref{smaller4}(d) we have $|P\cap N|=4$ and hence $|P|=8$ (note that $P$ is nonabelian as otherwise Theorem \ref{heightzero} and Theorem \ref{divisibility2} imply 4 divides 5). This contradicts the minimality of $G$ as a counterexample. Hence we may assume that $G_\tau=X$ for all non-trivial $\tau\in{\rm Irr}(B_0(X))$ and therefore 2 divides $\chi(1)$ for all $\chi\in{\rm Irr}(B_0)\setminus{\rm Irr}(G/X)$. Hence $k_0(B_0)=2$ and Theorem \ref{divisibility2} implies that $|G|_2=2$, a contradiction by Theorem \ref{smaller4}(b).

Suppose that $k(\bar B_0))=3$. Then $p=3$ and $|PN/N|=3$ by Theorem \ref{smaller4}(c). By Theorem \ref{divisibility3}, 3 divides $k_0(B_0)\leq 5$, and hence $3=k_0(B_0)$. By  Theorem \ref{divisibility3},  this implies that $|G|_3=3$. As $|PN/N|=3$, this forces $P\cap N=1$, a contradiction.

Suppose that $k(\bar B_0)=4$. Then all nontrivial irreducible characters in ${\rm Irr}(B_0(N))$ lie under the same irreducible character of ${\rm Irr}(B_0)$ and hence are $G$-conjugate. This forces $t=1$.

\smallskip

By Steps 5 and 7, the unique minimal normal subgroup $N$ of $G$ is nonabelian simple and $\cent G N=1$. Hence $G$ is almost simple and $N={\rm soc}(G)$. Note that in the case where $G/N$ is a $p'$-group, we have that $N \cent G P <G$ by Theorem \ref{isomblocks} together with the minimality of $G$ as a counterexample. Otherwise $k(\bar B_0)=4$ by Step 7.

\smallskip

Write $S={\rm soc}(G)$.  We have seen that $G/S$ is either a nontrivial $p'$-group or its principal $p$-block $\bar B_0$ has exactly 4 irreducible characters.
Using \cite{GAP} and the  GAP Character Table Library, one can check that if $S$ is a sporadic group or a simple alternating group ${\sf A}_n$ with $n\leq 6$, then $G$ is not a counterexample to the statement of Theorem A.
Suppose that $S\cong {\sf A}_n$ is a simple alternating group with $n> 6$.  In such cases $|G:S|\leq 2$. In particular $k(\bar B_0)\neq 4$ and
hence $|G:S|=2$ as $S<G$. Then $G\cong {\sf S}_n$ and $p$ is odd. In this case $P\in \syl p S$. By \cite[Proposition 4.10]{Ols90}, we have that
$k(B_0(S))\leq k(B_0)=5$. By minimality of $G$ as a counterexample, we actually have that $k(B_0(S))<5$. By Theorem \ref{smaller4}, either
$p=3$ and $P\cong {\sf C}_3$, which is absurd as $n>6$, or $p=5$ and $P\cong{\sf C}_5$, contradicting the choice of $G$ as a counterexample.
\end{proof}

\qquad

\section{Simple Groups of Lie Type}\label{sec:lie}
\noindent In this section, we prove that finite almost simple groups with socle a group of Lie type do not provide counterexamples to the statement of Theorem A. 
In view of Theorem \ref{thmB}, that will complete the proof of Theorem A.

\subsection{Preliminaries}
First, we check that the almost simple groups extending several small groups of Lie type satisfy Theorem \ref{thmA}.

\begin{prop}\label{prop:sporadics}
Let $G$ be an almost simple group such that its simple socle $S$ is $\tw{2}\type{F}_4(2)'$,  $\type{F}_4(2)$, $\type{G}_2(3)$,  $\type{G}_2(4)$,  $\PSL_2(8)\cong \tw{2}\type{G}_2(3)'$, $\type{B}_3(3)$, $\tw{2}\type{B}_2(8)$, $\tw{2}\type{E}_6(2)$, $\Sp_4(4)$, $\Sp_6(2)$, $\type{D}_4(2)$, $\PSU_6(2)$, $\PSL_3(2)$, $\PSL_3(4)$, $\PSp_4(3)\cong \PSU_4(2)$, $\PSU_4(3)$, or $\PSL_2(7)$.  Then Theorem \ref{thmA} holds for $G$ for all primes $p$ dividing  $|G|$.  Further, Theorem \ref{thmA} holds for $G$ if $p=2$ and $S$ is one of  $\Sp_4(8)$ or    $\PSU_5(4)$ and for the prime $p=3$ if $S$ is $\type{D}_4(3)$.  
\end{prop}
\begin{proof}
This can be seen using  \cite{GAP}.
\end{proof}

Now, given an almost simple group $G$ with socle $S$, with Lemma \ref{lem:blockabove} in mind, our strategy for proving Theorem \ref{thmA} in most cases will be to demonstrate at least $6$ non-$\aut(S)$-conjugate members of $B_0(S)$.  For $N\lhd G$, we write $k_G(B_0(N))$ for the number of distinct $G$-classes of characters in $B_0(N)$.  
Note that $k(B_0(G))\geq k_G(B_0(N))$, using Clifford theory and Lemma \ref{lem:blockabove}.

We next address some immediate cases that further reduce the situation, when combined with Theorem \ref{thmB}.

\begin{lem}\label{lem:cyclic}
Let $G$ be a finite group and $p$ a prime dividing $|G|$ such that $P\in \Syl_p(G)$ is cyclic.  Then $G$ is not a counterexample to Theorem \ref{thmA} for the prime $p$. 
\end{lem}
\begin{proof}
Write $B_0:=B_0(G)$. According to \cite[Theorem 5.1.2(ii)]{cravenbook}, we have 
\[k(B_0)=e+\frac{|P|-1}{e}\] where $e:=|\norm{G}{P}:\cent{G}{P}|$. From this we see that if $|P|\geq 8$, then $k(B_0)>5$, and if $|P|\leq 4$, then  $k(B_0)\leq 4$.  In either of these cases, $G$ is then not a counterexample to Theorem \ref{thmA}.  If $|P|\in\{5,7\}$, then $P\in \{\cyc{C}_5, \cyc{C}_7\}$, and hence $G$ is again not a counterexample to Theorem \ref{thmA}, completing the claim.
\end{proof}

\begin{lem}\label{lem:quot4}
Assume that $G$ is almost simple with socle $S$ and that $k(B_0(G/S))=4$.  If $k_G(B_0(S))\geq 3$, then $G$ is not a counterexample to Theorem \ref{thmA}. 
\end{lem}

\begin{proof}
Note that by assumption, $B_0(G)$ contains at least 2  characters that are nontrivial on $S$, using Lemma \ref{lem:blockabove}.  But $k(B_0(G/S))=4$ implies there are 4 additional characters in $B_0(G)$ that are trivial on $S$, since $B_0(G/S)\subseteq B_0(G)$.  Hence $k(B_0(G))\geq 6$.
\end{proof}

%


\subsection{Additional Notation}
To ease our notation, we will switch to using $A$ for the remainder of the paper to denote an almost simple group.  We will let $G$ be a quasisimple group of Lie type with $S=G/\zent{G}$, where $\zent G$ is a non-exceptional Schur multiplier for $S$ and $G$ is defined over $\mathbb{F}_q$, where $q$ is a power of the prime $q_0$.  (Note that with Proposition \ref{prop:sporadics} and Theorem \ref{thmB}, we have already completed the case that $S$ has an exceptional Schur multiplier.  See e.g. \cite[Table 6.1.3]{GLS3} for the list of simple groups of Lie type with exceptional Schur multipliers.)  Then we may assume that $G$ is the set of fixed points $\bG^F$ of a simple, simply connected algebraic group $\bG$ defined over $\overline{\mathbb{F}}_q$, under a Steinberg endomorphism $F$.

Throughout, for $\epsilon\in\{\pm1\}$ (or simply $\{\pm\}$), the notation $\PSL_n^\epsilon(q)$ will denote $\PSL_n(q)$ for $\epsilon=1$ and $\PSU_n(q)$ for $\epsilon=-1$.  Similarly, $ \operatorname{P\Omega}_{2n}^\epsilon(q)$ will denote the simple group $\operatorname{P\Omega}_{2n}(q)$ of type $\type{D}_n(q)$ for $\epsilon=1$ and its twisted counterpart $\operatorname{P\Omega}_{2n}^-(q)$ of type $\tw{2}\type{D}_n(q)$ for $\epsilon=-1$.  An analogous meaning will be used for related groups such as $\GL_n^\epsilon(q)$, $\SL_n^\epsilon(q)$, $\GO_{2n}^\epsilon(q)$, and $\SO_{2n}^\epsilon(q)$ and for the groups $\type{E}_6^\epsilon(q)$ representing $\type{E}_6(q)$ and $\tw{2}\type{E}_6(q)$.

\subsection{Defining Characteristic}

In this section, we will assume $q_0=p$, so that $S$ is defined in the same characteristic as the blocks under consideration.  In this case, using \cite[Theorem 3.3]{cabanes17}, we see that there are exactly two $p$-blocks for $S$: $B_0(S)$ and a defect-zero block containing the Steinberg character.  
With this in mind, we have $k(B_0(S))=k(S)-1$. Combining with Lemma \ref{lem:blockabove}, for $S\leq A \leq \aut(S)$, it suffices to show $\frac{k(S)-1}{|A/S|}\geq 6$ when $P\in\Syl_p(A)$ is not one of $\{ {\sf C}_5, {\sf C}_7, {\sf D}_8, {\sf Q}_8 \}$.  Now, note further that $k(S)\geq k(G)/|\zent{G}|$,
giving a rough bound 
\begin{equation}\label{eq:crudebound}
k(B_0(A))\geq k_A(B_0(S))\geq \frac{k(G)-|\zent{G}|}{|\zent{G}||\out{(S)}|.}
\end{equation}

Further, \cite[Theorem 3.7.6]{carter} implies that the number of semisimple classes of $G$ is $|(\zent{\bg{G}}^\circ)^F|q^r$, where $r$ is the rank of $G$, which forces $k(G)>q^r$ since $G$ also contains non-semisimple classes.  
Combining this with \eqref{eq:crudebound}, we see that
\begin{equation}\label{eq:lesscrudebound}
k(B_0(A))>\frac{q^r-|\zent{G}|}{|\zent{G}||\out(S)|}.
\end{equation}

%
%
%
%
%

\begin{prop}\label{prop:definingchar}
Let $S$ be a simple group of Lie type defined in characteristic $p$ such that $S$  is not isomorphic to an alternating group nor one of the groups treated in Proposition \ref{prop:sporadics}. Then   $k(B_0(A))\geq 6$ for any $S\leq A\leq \aut(S)$, where $B_0(A)$ is the principal $p$-block of $A$.
\end{prop}
\begin{proof}
Recall that our assumptions assure $S=G/\zent{G}$, where $G$ is a group of Lie type in characteristic $p$ whose underlying algebraic group is simple and simply connected and $\zent{G}$ is a nonexceptional Schur multiplier for $S$.

First,  \cite{lubeckwebsite} contains the explicit values of $k(G)$ for $G$ of exceptional type $\tw{2}\type{B}_2(q)$, $\tw{2}\type{G}_2(q)$, $\type{G}_2(q),$ $\tw{2}\type{F}_4(q)$, $\type{F}_4(q)$, $\tw{3}\type{D}_4(q)$, $\tw{2}\type{E}_6(q)$, $\type{E}_6(q)$, $\type{E}_7(q)$, and $\type{E}_8(q)$.  Using the bound in \eqref{eq:crudebound}, together with this information and the knowledge of $|\zent{G}|$ and $|\out(S)|$ in each case, we see that the statement holds for these groups.    

We may therefore assume that $G$ is of classical type.  Let $q=p^a$. Table \ref{tab:boundsClassicalDefining} gives upper bounds on $|\zent{G}|$ and $|\out(S)|$ in each case.  
\begin{table}\caption{}
\tiny
\begin{tabular}{|c|c|c|}
\hline
$S$ & Upper Bound on $|\zent{G}|$ & Upper Bound on $|\out(S)|$\\
\hline\hline
$\PSL_n^\epsilon(q)$ & $p^a-\epsilon$ & $2(p^a-\epsilon)a$\\
\hline
$\operatorname{P\Omega}_{2n}^\epsilon(q)$  & 4 & $8a$\\
with $n\geq 4$, $q$ odd, and $(n,\epsilon)\neq (4,+1)$ & & \\
\hline
$\operatorname{P\Omega}_{2n}^\epsilon(q)$  & 1 & $2a$\\
with $n\geq 4$, $p=2$, and $(n,\epsilon)\neq (4,+1)$ & & \\
\hline
$\operatorname{P\Omega}_8(q)$, $q$ odd & 4 & $24a$\\
\hline
$\operatorname{P\Omega}_8(q)$, $p=2$ & 1 & $6a$\\
\hline
$\PSp_{2n}(q)$ or $\operatorname{P\Omega}_{2n+1}(q)$ & 2 & $2a$ \\
with $n\geq 2$ and $(n,p)\neq (2,2)$ & & \\
\hline
$\Sp_{4}(2^a)$  & 2 & $4a$ \\
\hline
\end{tabular}\label{tab:boundsClassicalDefining}

\end{table}
We see using these bounds and \eqref{eq:lesscrudebound} that $k(B_0(A))\geq 6$ except possibly in the cases 
$\PSp_4(5)$, $\Sp_8(2)$, $\operatorname{P\Omega}_8^-(3)$, $\PSU_5(2)$, or $\PSL_n^\epsilon(q)$ with $n\leq 4$.
(Recall that $S$ is not as in Proposition \ref{prop:sporadics}.) It can be readily checked in GAP that $k(B_0(S))\geq 6$ if $S$ is one of the first four in this list.  Further,  using the character tables in  \cite{chevie}, we see that, since the cases with exceptional Schur multiplier are omitted here, the simple groups $\PSL_3^\epsilon(q)$ and $\PSL_4^\epsilon(q)$ under consideration have at least 6 distinct character degrees other than the Steinberg character, showing that $k_A(B_0(S))\geq 6$ in this case when combined with Lemma \ref{lem:quotientblock} and recalling that $k(B_0(S))=k(S)-1$.

So, we now assume $S=\PSL_2(q)$. Then $\out(S)$ is generated by the action of $\PGL_2(q)$ and the field automorphisms. The character tables for $\PSL_2(q)$, $\SL_2(q)$, and $\PGL_2(q)$ are well-known.  First let $q$ be odd and write $q\equiv \eta\pmod 4$ with $\eta\in\{\pm1\}$.  Then we have two characters of degree $\frac{q+\eta}{2}$ that are switched by the action of $\PGL_2(q)$ and fixed by the field automorphisms. For $\eta=1$, there are $\frac{q-5}{4}$ characters of degree $q+1$ and $\frac{q-1}{4}$ characters of degree $q-1$.  For $\eta=-1$, there are $\frac{q-3}{4}$ each of characters of degree $q-1$ and $q+1$.  These  characters are fixed by the action of $\PGL_2(q)$ but permuted by the field automorphisms.  Recall that $q=p^a$, so the size of $\out (S)$ is $2a$.  Then in the case that $p$ is odd, this yields $k_{\aut(S)}(B_0(S))\geq 2+\frac{q-3}{2a}=2+\frac{p^a-3}{2a}$. This is larger than 5 for $q\geq 11$.  Then by our assumption that $S$ is not solvable, alternating, or as in Proposition \ref{prop:sporadics}, we are done in this case.

Finally, assume $p=2$. Then $S$ has $\frac{q}{2}-1$ irreducible characters of degree $q+1$ and $\frac{q}{2}$ irreducible characters of degree $q-1$.  Here $\out(S)$ is cyclic of size $a$, where $q=2^a$.  This yields $k_{\aut(S)}(B_0(S))\geq 1+\frac{2^a-1}{a}$, which is larger than 5 for $a\geq 5$.  Further, the statement can be checked using GAP when $S=\PSL_2(16)$, completing the proof. 
\end{proof}

\subsection{Non-Defining Characteristic}
From now on, let $q$ be a power of a prime $q_0$ different from $p$ and let $d_p(q)$ be the order of $q$ modulo $p$ if $p$ is odd, and let $d_2(q)$ be the order of $q$ modulo $4$.

 In this situation, it will often be useful to consider a certain collection of characters of a simple group $S$ of Lie type known as \emph{unipotent characters}.  A block containing a unipotent character is called a unipotent block, and in particular the principal block is one such unipotent block.

 If $G$ is a group of Lie type such that the underlying algebraic group is simply connected and such that $S=G/\zent{G}$, and $\wt{G}$ is a group of Lie type such that the underlying algebraic group has a connected center and $G\lhd \wt{G}$ via a regular embedding as in \cite[Section 15.1]{CE04}, then  the work of Lusztig \cite{lus88} (see also \cite[2.3.14 and 2.3.15]{GM20}) shows that the unipotent characters of $\wt{G}$ are irreducible on restriction to $G$ and trivial on $\zent{\wt{G}}$ (and hence on $\zent{G}$).  Therefore, unipotent characters in $B_0(\wt{G})$ may also be viewed as unipotent characters of $S$ lying in $B_0(S)$ by applying Lemma \ref{lem:quotientblock}. Further,  Lusztig has determined the stabilizers in $\aut({S})$ of the unipotent characters, which is also summarized in \cite[Theorem 2.5]{malle08}.  In particular, the unipotent characters are stabilized by $\aut(S)$ except in the following cases: $\type{B}_2(2^{2n+1})$, $\type{G}_2(3^{2n+1})$, and $\type{F}_4(2^{2n+1})$, in which cases the exceptional graph automorphism permutes certain pairs of unipotent characters; $\type{D}_4(q)$, in which case the graph automorphism of order three permutes two specific triples of unipotent characters, and $\type{D}_n(q)$ with even $n\geq 4$, in which case the graph automorphism of order $2$ interchanges the pairs of unipotent characters labeled by so-called degenerate symbols.

The set $\irr{\wt{G}}$ can be partitioned into so-called Lusztig series $\mathcal{E}(\wt{G}, s)$, which are indexed by $\wt{G}^\ast$-classes of semisimple elements $s$ in a group $\wt{G}^\ast$, known as the dual group of $\wt{G}$.  The elements of $\mathcal{E}(\wt{G},s)$ are further in bijection with unipotent characters of $\cent{\wt{G}^\ast}{s}$, with $1_{\cent{\wt{G}^\ast}{s}}$ corresponding to the so-called \emph{semisimple} character $\chi_s$.   We record the following, which is a specific case of \cite[Theorem 9.12]{CE04}.
 
 \begin{lem}\label{lem:nondefblocks}
 With the notation above, the characters lying in the unipotent $p$-blocks of $\wt{G}$ are exactly those characters in Lusztig series $\mathcal{E}(\wt{G}, t)$ indexed by semisimple  $p$-elements $t$ in  $\wt{G}^\ast$.   Further, $B_0(\wt{G})$ is the unique unipotent block of $\wt{G}$ if and only if all unipotent characters lie in $B_0(\wt{G})$. 
 \end{lem}
 
We also record  the following observation, which has been useful in many contexts for determining elements of $\irrp{p}{B_0}$. 
 
 \begin{lem}\label{lem:uniquep'dregular}
 Let $G=\bG^F$ be a group of Lie type defined over $\mathbb{F}_q$ such that the underlying algebraic group $\bg{G}$ is simple and simply connected and such that $G$ is not of Suzuki or Ree type.  Let $p\nmid q$ be a prime such that $d:=d_p(q)$ is a regular number for $G$ in the sense of Springer \cite{springer74}.  Then $B_0(G)$ is the unique block of $G$ containing unipotent characters with $p'$-degree.
 \end{lem}
 \begin{proof}
By \cite[Corollary 6.6]{malle07}, any unipotent character of $p'$-degree lies in a $d$-Harish-Chandra series indexed by $(L, \lambda)$ where $L$ is the centralizer $\cent{G}{\bg{S}}$ of a Sylow $d$-torus $\bg{S}$ of $\bG$.  Further, $\cent{\bG}{\bg{S}}$ is a torus since $d$ is regular (see \cite[Definition 2.5]{spath09}), and hence is further a maximal torus since $\cent{\bG}{\bg{S}}$ necessarily contains a maximal torus.  Hence there is a unique such series.  By \cite[Theorem A]{enguehard00}, all members of this series lie in the same block of $G$, which is therefore the principal block $B_0(G)$.
 \end{proof}

We begin by considering the simple exceptional groups of Lie type, by which we mean  $\tw{2}\type{B}_2(2^{2n+1})$, $\type{G}_2(q)$, $\tw{2}\type{G}_2(3^{2n+1})$, $\type{F}_4(q)$, $\tw{2}\type{F}_4(2^{2n+1})$, $\tw{3}\type{D}_4(q)$, $\type{E}_6^\pm(q)$, $\type{E}_7(q)$, and $\type{E}_8(q)$.  In the cases of Suzuki and Ree groups, we only consider $n\geq 1$, as the Tits group $\tw{2}\type{F}_4(2)'$ and the group $\tw{2}\type{G}_2(3)'\cong \type{A}_1(8)$ are dealt with in Proposition \ref{prop:sporadics}.

\begin{lem}\label{lem:exceptionalcross}
Let $S$ be an exceptional group of Lie type  defined over $\mathbb{F}_q$ as above, and let $p$ be a prime not diving $q$.  Let $A$ be an almost simple group with simple socle $S$.  Then $A$ is not a minimal counterexample to Theorem \ref{thmA}.
\end{lem}

\begin{proof} 
For $S=\tw{2}\type{B}_2(2^{2n+1})$ or $S=\tw{2}\type{G}_2(3^{2n+1})$, the Sylow $p$-subgroups are cyclic if $p\geq 5$. (And note that $3\nmid |S|$ in the first case.) When $p\geq 5$, we further see that $k_A(B_0(S))\geq 3$ since by \cite{BurkhardtSz}, $B_0(S)$ contains 3 characters of distinct degrees for $\tw{2}\type{B}_2(2^{2n+1})$,  and by \cite{Hiss91}, there are at least 3 characters in $B_0(S)$ in the case $\tw{2}\type{G}_2(3^{2n+1})$, of which at least 2 are unipotent. Then we are done in this case by  applying Theorem \ref{thmB} and Lemmas \ref{lem:cyclic} and \ref{lem:quot4}.  
For $p=2$ and $S=\tw{2}\type{G}_2(3^{2n+1})$, there are at least 6 characters of distinct degree in $B_0(S)$ by \cite{ward}, forcing $k_A(B_0(S))\geq 6$.    
For $S=\tw{2}\type{F}_4(2^{2n+1})$ and $p\geq 3$, we see using \cite{malle91} that there are at least 6 unipotent characters in $B_0(S)$, so that again $k_A(B_0(S))\geq 6$.  

Hence we may assume that $S$ is not a Suzuki or Ree group.  We claim that we may further assume that $d_p(q)$ is a regular number, in the sense of Springer \cite{springer74}.  If $d_p(q)$ is a non-regular number, we see from \cite[Table 2]{BMM} combined with \cite[Theorem A]{enguehard00} that there are at least 6 unipotent characters in $B_0(S)$, with the possible exception of $\type{E}_6(q)$ when $d_p(q)=5$ or $\tw{2}\type{E}_6(q)$ when $d_p(q)=10$, in which cases \cite[Table 2]{BMM} instead yields at least 5 unipotent characters in $B_0(S)$. Recalling that these unipotent characters are invariant under $\aut(S)$, this gives $k_A(B_0(S))\geq 6$, respectively 5.  In the cases $\type{E}_6(q)$ when $d_p(q)=5$ or $\tw{2}\type{E}_6(q)$ when $d_p(q)=10$, a Sylow $p$-subgroup of $S$ is cyclic, and hence $A$ is not a minimal counterexample to Theorem \ref{thmA} by again applying Theorem \ref{thmB} with Lemmas \ref{lem:cyclic} and \ref{lem:quot4}.

We now assume $d:=d_p(q)$ is a regular number.   Let $G$ be a group of Lie type of simply connected type such that $G/\zent{G}=S$.  By Lemma \ref{lem:uniquep'dregular}, the principal block $B_0(G)$ is the unique block of $G$ containing unipotent characters of $p'$-degree.  Then from the discussion above, it would suffice to show there exist at least $6$ unipotent characters of $p'$ degree for $G$ (hence $S$) that are not conjugate under  $\aut(S)$. Recall that in the cases under consideration, the unipotent characters are $\aut(S)$-invariant except for the case of $\type{G}_2(3^{2m+1})$ and $\type{F}_4(2^{2m+1})$.

 For $G$ of type $\type{E}_6^\pm(q)$, $\type{E}_7(q)$, or $\type{E}_8(q)$, we see from the list of unipotent character degrees in \cite[Section 13.9]{carter} that there are at least 6 $p'$-degree unipotent characters for each regular $d$, so we are done in these cases. 
 
 For $\type{F}_4(q)$, again using the list in \cite[Section 13.9]{carter}, we see for $d>2$ or $p>3$ that there are at least 6 $p'$-degree unipotent characters with distinct degrees.  Further, there are at least 6 $p'$-degree unipotent characters for $p=2,3$.  In the case $p=3$ and $q$ is an odd power of $2$, we have the trivial character, the Steinberg character, and at least four pairs of $3'$-degree unipotent characters, where the pairs are permuted by the graph automorphism (see \cite[Theorem 2.5]{malle08}) but left invariant by all other members of $\aut(S)$.  This still yields at least 6 $\aut(S)$-orbits of unipotent characters of $p'$-degree in $B_0(S)$, so that $k_A(B_0(S))\geq 6$.

For $S=\type{G}_2(q)$, \cite{HS92, HS90} show that there are at least 6 characters in $B_0(S)$ for $p=2$ and $p=3$, respectively, with distinct degrees.  For $d_p(q)\geq 3$, we see using \cite[Section 13.9]{carter} that there are again at least 6 $p'$-degree unipotent characters, which are further not permuted by the exceptional graph automorphism if $q$ is an odd power of $3$.  For $d_p(q)\in\{1,2\}$  and $p>3$, there are six $p'$-degree unipotent characters. In the case that $q$ is an odd power of $3$, two of these characters  are interchanged by the exceptional graph automorphism. Hence in the latter case, there are 5 different $\aut(S)$-classes of unipotent characters in $B_0(S)$.   But there are also non-unipotent characters in $B_0(S)$ (using e.g. \cite{Shamash92}), yielding $k_{\aut(S)}(B_0(S))\geq 6$.

Finally, if $S=\tw{3}\type{D}_4(q)$, \cite{DeM87} shows that there are at least 6 characters in $B_0(S)$ with distinct degrees for $p=2,3$.   If $d\in\{3,6\}$ or $d\in\{1,2\}$ with $p>3$, \cite[Section 13.9]{carter} shows that there are at least 6 $p'$-degree unipotent characters. Hence we see $k_{\aut{S}}(B_0(S))\geq 6$, except possibly for $d=12$. If $d=12$, there are four unipotent characters of $p'$-degree, and hence $k_{\aut(S)}(B_0(S))\geq 4$.  But in this case, a Sylow $p$-subgroup is cyclic, so we are again done by applying Theorem \ref{thmB} and Lemmas \ref{lem:cyclic} and \ref{lem:quot4}. 
\end{proof}

In the context of Theorem \ref{thmB}, we have now further reduced ourselves to the case that $S$ is a finite classical group defined in a characteristic different from $p$.  In the remaining sections, we address these cases.

\subsubsection{Linear and Unitary Groups with $p$ odd}

Let $S=\PSL_n^\epsilon(q)$ and write $G=\SL_n^\epsilon(q)$ and $\wt{G}=\GL_n^\epsilon(q)$, where $q$ is a power of some prime. 
 Let $p$ be an odd prime not dividing $q$ and let $\wt{B}_0$ be the principal $p$-block of $\wt{G}$.  
Recall from before that unipotent characters in $\wt{B}_0$ may also be viewed as characters of $B_0(S)$ and that they are invariant under $\aut(S)$.   In most cases, we aim to show that $\wt{B}_0$ contains at least 6 unipotent characters, which will force $k(B_0(A))\geq k_A(B_0(S))\geq 6$ for any almost simple group $A$ with simple socle $S$.

Let $e$ be the order of $q$ modulo $p$ if $\epsilon=1$ and of $q^2$ modulo $p$ if $\epsilon=-1$, and let $e'$ be as follows: 
\[e':=\left\{\begin{array}{cc} 
e & \hbox{ if $\epsilon=1$}\\
e & \hbox{if $\epsilon=-1$ and $p\mid q^e-(-1)^e$}\\
2e & \hbox{ if $\epsilon=-1$ and $p\mid q^e+(-1)^e$}.\\
\end{array}\right.\] 
Write $n=me'+r$ with $0\leq r<e'$.   The unipotent characters of $\wt{G}$ are indexed by partitions of $n$, and two such characters lie in the same block exactly when the corresponding partitions have the same $e'$-core, by \cite{FS82}. 
Identifying the trivial character with the partition $(n)$,  a unipotent character is then contained in $\wt{B}_0$ if and only if the corresponding partition has $e'$-core $(r)$.

Now, by \cite[Theorem 1.9]{michlerolsson}, $\wt{B}_0$ has the same block-theoretic invariants as the principal $p$-block $B_{me'}$ of $\GL^\epsilon_{me'}(q)$.  
Here, unipotent characters of $\GL_{me'}^\epsilon(q)$ lie in $B_{me'}$ exactly when they have trivial $e'$-core.  Further, there is a bijection between partitions of $me'$ with trivial $e'$-core and partitions of $n$ with $e'$-core equal to $(r)$.  That is, there is a bijection between unipotent characters in $\wt{B}_0$ and unipotent characters in $B_{me'}$.  By \cite[Theorem 3.2]{BMM}, the number of unipotent characters in $\wt{B}_{0}$ is further given by the number of irreducible characters of the relative Weyl group of a Sylow $e$-torus of $\wt{G}$. We also remark that a Sylow $p$-subgroup is cyclic in the case $m=1$.

\begin{lem}\label{typeAmost}
Let $A$ be an almost simple group with socle $S=\PSL_n^\epsilon(q)$, where $q$ is a power of a prime $q_0$.   Let $p\neq q_0$ be an odd prime. With the notation above, $A$ is not a counterexample to Theorem \ref{thmA} as long as $me'\geq 5$. 
\end{lem}
\begin{proof}
Given the discussion above, we know that if $B_{me'}$ contains at least 6 unipotent characters, then $k_A(B_0(S))\geq 6$. 
 But by taking partitions of the form $(1^c, me'-c)$ for $0\leq c\leq me'$, we see that $B_{me'}$ contains at least 6 unipotent characters as long as $me'\geq 6$.  

Now, assume $me'=5$.  Then either $e'=1$ or $e'=5=me'$.  In the first case, all unipotent characters lie in $B_{me'}$, so $B_{me'}$ again contains at least 6 unipotent characters.  If $e'=me'=5$, then $B_{me'}$ contains $5$ unipotent characters, and hence $B_0(S)$ contains at least $5$ characters that are $A$-invariant.  
Further, in this case $\GL_{me'}^\epsilon(q)$ has a cyclic Sylow $p$-subgroup, and hence so does $\wt{G}$. The argument in Lemma \ref{lem:cyclic} shows that $\wt{B}_0$ contains at least one more character, which is necessarily not $A$-conjugate to the 5 unipotent characters. 
   Now, by Lemma \ref{lem:nondefblocks}, the non-unipotent characters in $\wt{B}_0$ must be in Lusztig series $\mathcal{E}(\wt{G}, t)$ with $t\in\wt{G}^\ast\cong\wt{G}$ a nontrivial $p$-element.  Note that since $n\geq 5$, we have $|\zent{\wt{G}^\ast}|=|\wt{G}^\ast/ \Oh{q_0'}{\wt{G}^\ast}|=|\wt{G}^\ast/[\wt{G}^\ast, \wt{G}^\ast]|=| \wt{G}/G|$, and since $e'=5$ and hence $p\nmid (q-\epsilon)$, we know this number is not divisible by $p$.  Then using \cite[Lemma 2.6]{SFT20}, we have every member of such a series $\mathcal{E}(\wt{G}, t)$ is trivial on the center and cannot lie above a unipotent character of $G$.  This forces a sixth member of $B_0(S)$ that is not $A$-conjugate to the unipotent characters.
\end{proof}

Now we are ready to complete the case that $S=\PSL_n^\epsilon(q)$ with $p\nmid q$ an odd prime. 

\begin{prop}\label{prop:typeAsmall}
Let $A$ be an almost simple group with socle $S=\PSL_n^\epsilon(q)$ where $n\geq 2$ and $q$ is a power of a prime, and  let $p$ be an odd prime not dividing $q$. Then $A$ is not a minimal counterexample to Theorem \ref{thmA}. 
\end{prop}
\begin{proof}
By Theorem \ref{thmB}, we may assume that
either $p$ does not divide $|A:S|$ and $S\cent A P <A$, or that $k(B_0(A/S))=4$.  
Keeping the notation from above and applying Lemma \ref{typeAmost}, we may also assume $me'\leq 4$.  Let $P$ be a Sylow $p$-subgroup of $S$.

 First, let $me'=4$. Then $B_{me'}$ contains at least 4 unipotent characters, taking the forms $(1^c, me'-c)$ as before.  In particular, $k_A(B_0(S))\geq 4$, and we may assume $p\nmid [A:S]$, by Lemma \ref{lem:quot4}.  If $e'=1$ or $2$, we get one additional unipotent character in $B_{me'}$, corresponding to $(2,2)$. Hence in these cases, it suffices to note as before that there is at least one non-unipotent character in $B_{me'}$ (and hence in $\wt{B}_0$), which we may choose to further be trivial on the center and  lie above a non-unipotent character in $B_0(S)$ by choosing a character in $\mathcal{E}(\wt{G},t)$ where $t\in [\wt{G}^\ast, \wt{G}^\ast]\cong G$  is a non-central $p$-element.     This leaves the case $me'=4=e'$, in which case we must have $p\geq 5$ and $P$ is cyclic. Hence, we are done by Lemma \ref{lem:cyclic}.

   Now, if $me'= 3$, note that there are three unipotent characters in $B_{me'}$, and hence in $B_0(S)$.  If $me'=2$, there are two unipotent characters in $B_{me'}$, though we may find a third character in $B_0(S)$ by taking $t\in[\wt{G}^\ast, \wt{G}^\ast]$ to have eigenvalues $\{a, a^{-1}\}$ with $|a|=p$ and arguing as before.  In any case, we see $k_A(B_0(S))\geq 3$, so we may again assume $p\nmid [A:S]$.  Then by Lemma \ref{lem:cyclic}, we may assume $n=3$ and $p\mid (q-\epsilon)$, since in the other cases for $me'\leq 3$, we have $P$ is cyclic.  Note that in the  case $n=3$ and $p\mid (q-\epsilon)$, all three unipotent characters lie in the principal block of $B_0(\wt{G})$.

   If $p=3$, the semisimple character $\chi_t$ of $\wt{G}$ indexed by a semisimple element with eigenvalues $\{a, a^{-1}, 1\}$ with $|a|=3$ is trivial on $\zent{\wt{G}}$ by \cite[Lemma 2.6]{SFT20} 
   since $t\in[\wt{G}^\ast, \wt{G}^\ast]$ and lies in $B_0(\wt{G})$ by Lemma \ref{lem:nondefblocks}.  Further, \cite[Lemma 2.6]{SFT20} and Lemma \ref{lem:restnumber} also imply that since  $t$ is $\wt{G}^\ast$-conjugate to $tz$ where $z\in \zent{\wt{G}}$ has order 3, we have $\chi_{tz}=\chi_t$ restricts to the sum of 3 irreducible characters in $G$, which must lie in $B_0(G)$ (and hence $B_0(S)$).  This yields $k(B_0(S))\geq 6$.  But since $3\nmid [A:S]$, we further have that at least one of these three characters must be invariant under $A$, and hence $k_A(B_0(S))\geq 5$.  Now, these characters (and those above them in $A$) have height zero, and the same is true for the unipotent characters.  Hence it follows from Theorem \ref{divisibility3} that in fact $k(B_0(A))\geq 6$.

   Finally, assume that $n=3$,  $p\mid (q-\epsilon)$, and  $p> 3$.  Recall that $B_0(\wt{G})$ is the unique unipotent block of $\wt{G}$. Let $t_1, t_2\in[\wt{G}^\ast, \wt{G}^\ast]$ be such that $t_1$ has eigenvalues $\{a, a^{-1}, 1\}$ and $t_2$ has eigenvalues $\{a, a, a^{-2}\}$, where $|a|=p$. Then the characters in the series $\mathcal{E}(\wt{G}, t_i)$ for $i=1,2$ are trivial on $\zent{\wt{G}}$ and lie in $B_0(\wt{G})$ by  \cite[Lemma 2.6]{SFT20} and Lemma \ref{lem:nondefblocks}.  Further, note that $t_i$ cannot be $\aut({S})$-conjugate to $t_jz$ for any $1\neq z\in Z(\wt{G}^\ast)$ or $j\in\{1,2\}$, and hence the characters in $\mathcal{E}(\wt{G}, t_1)$ are not $\aut({S})$-conjugate to those in $\mathcal{E}(\wt{G}, t_2)$ and restrict irreducibly to $G$ by 
   Lemma \ref{lem:restnumber}. Hence we see that the characters in these series may be viewed as members of $B_0(S)$, yielding $k_A(B_0(S))\geq 5$ when combined with the unipotent characters in the block.  Finally, since $\cent{\wt{G}^\ast}{t_2}\cong X_1\times X_2$ with $X_1\in\{ \GL_2^\pm(q)\}$ and $X_2\in\{ \GL_1^\pm(q)\}$, it follows that there are two members of  $\mathcal{E}(\wt{G}, t_2)$ with distinct degrees, so $k_A(B_0(S))\geq 6$.  
 \end{proof}

\subsubsection{Remaining Classical Groups with $p$ odd}
  
We set some notation to be used throughout this section.

 Let $q$ be a power of some prime and let $S$ be a simple group $\type{B}_n(q)$ with $n\geq 3$, $\type{C}_n(q)$ with $n\geq 2$, or $\type{D}_n(q)$ or $\tw{2}\type{D}_n(q)$ with $n\geq 4$.  Let $p\nmid q$ be an odd prime and let $e:=d_p(q)/\gcd(2,d_p(q))$ where $d_p(q)$ is the order of $q$ modulo $p$. Write $n=me+r$, where $0\leq r<e$ is the remainder when $n$ is divided by $r$.

Let $H$ be the corresponding symplectic or special orthogonal group  $\operatorname{SO}_{2n+1}(q)$, $\Sp_{2n}(q)$, or $\operatorname{SO}^\epsilon_{2n}(q)$.  For the cases of special orthogonal groups with $q$ odd, let $\Omega\leq H$ be the unique subgroup of index 2, and otherwise let $\Omega:=H$, so that $\Omega/\zent{\Omega}=S$.  
 Further, let $\overline{H}$ be the group $\operatorname{GO}^\epsilon_{2n}(q)$ in the case of type $\type{D}_n, \tw{2}\type{D}_n$, and otherwise let $\overline{H}:=H$.

\begin{lem}\label{classicalmost}
Let $A$ be an almost simple group with socle $S$ of type $\type{B}_n(q)$ with $n\geq 3$, $\type{C}_n(q)$ with $n\geq 2$, or $\type{D}_n(q)$ or $\tw{2}\type{D}_n(q)$ with $n\geq 4$, where $q$ is a power of a prime $q_0$.  Let   $p\neq q_0$ be an odd prime. Then $A$ is not a counterexample to Theorem \ref{thmA} as long as $me\geq 3$. 
\end{lem}
 \begin{proof}
 Keep the notation from above.   Using the main Theorem of \cite{CE94} to argue as in \cite[Discussions before Props 5.4 and 5.5]{malle17},  the number of unipotent characters in $B_0(\overline{H})$  is $k(2e, m)$, where the number $k(2e, m)$ may be computed as in \cite[Lemma 1]{olsson84}.  

Now, similar to before, results of Lusztig give that the unipotent characters of $H$ restrict irreducibly to $\Omega$ and are trivial on $\zent{\Omega}$ (see \cite[2.3.14 and 2.3.15]{GM20}), and hence can be viewed as irreducible characters of $S$.  Recall that for $S\neq \type{D}_4(q)$ nor $\PSp_4(2^{2a+1})$, the only automorphisms of $S$ that do not fix the unipotent characters occur in the case of type $\type{D}_n$.  In the latter case,  the graph automorphism, induced by the action of $\overline{H}$ on $H$, interchanges the pairs of unipotent characters of $H$ parameterized by so-called degenerate symbols. Further, unipotent characters of $\overline{H}$ are defined as the characters lying above unipotent characters of $H$, and tensoring by the nontrivial linear character of $\overline{H}/H$ yields another block above $B_0(H)$. Then we see in this case that $k(2e,m)$ describes the number of $\overline{H}$-conjugacy classes of unipotent characters in $B_0(H)$ (and hence of $B_0(S)$).  Therefore, if $S\neq \type{D}_4(q)$ nor $\PSp_4(2^{2a+1})$, $B_0(S)$ contains $k(2e,m)$ non-$\aut({S})$-conjugate unipotent characters, and hence $k_A(B_0(S))\geq k(2e,m)$ for any $S\leq A\leq \aut({S})$.  Using \cite[Lemma 1]{olsson84}, we may calculate that $k(2e, m)\geq 6$ unless $em=2$.

For $S=\type{D}_4(q)$, the above is still true, except possibly if the unipotent characters permuted by the exceptional graph automorphism of order $3$ lie in the principal block. Using the theory of $e$-cores and $e$-cocores of \cite{FS89}, we see that this only happens when $e=1$ or $2$, in which case we may again use \cite[Lemma 1]{olsson84} to see that $k(2e,m)> 11$, so there are still at least 6 non-$\aut({S})$-conjugate unipotent characters in $B_0(S)$.  
 \end{proof}

 \begin{prop}\label{prop:classicalsmall}
 Let $A$ be an almost simple group with socle $S$ of type $\type{B}_n(q)$ with $n\geq 3$, $\type{C}_n(q)$ with $n\geq 2$, or $\type{D}_n(q)$ or $\tw{2}\type{D}_n(q)$ with $n\geq 4$, where $q$ is a power of a prime.  Let $p$ be an odd prime not dividing $q$. Then $A$ is not a minimal counterexample to Theorem \ref{thmA}. 
 \end{prop} 
 \begin{proof}  
  Note that by Lemma \ref{classicalmost}, we may assume that $em=2$, and hence $n=2$ or $n=3$, so $S$ is type $\type{B}$ or $\type{C}$.
  If $e=2$, then a Sylow $p$-subgroup of $S$, $H$, or $\Omega$ is cyclic.   Further, in this case the number of unipotent characters in $B_0(S)$ is $4$.  Note that for $\PSp_4(2^{2a+1})$, the graph automorphism interchanges two unipotent characters, but in any case we still have $k_A(B_0(S))\geq 3$ for any $S\leq A\leq \aut({S})$.  Then using Lemmas \ref{lem:cyclic} and \ref{lem:quot4} and Theorem \ref{thmB}, we see $A$ is not a minimal counterexample for any $S\leq A\leq \aut({S})$.   If $e=1$, we have $n=2$ and $S=\PSp_{4}(q)$.  In this case, we see from \cite{white90b, white92, white95} that $k_{\aut({S})}(B_0(S))\geq 6$, completing the proof.
\end{proof}

\subsubsection{Classical Groups with $p=2$}

\begin{lem}\label{prime2classicals}
Let $q$ be a power of an odd prime $q_0$ and let $p=2$.  Let $A$ be an almost simple group with simple socle $S=\PSL_n^\epsilon(q)$ with $n\geq 3$, $\PSp_{2n}(q)$ with $n\geq 2$, $\operatorname{P\Omega}_{2n+1}(q)$ with $n\geq 3$ or $\operatorname{P\Omega}_{2n}^\epsilon(q)$ with $n\geq 4$. Then $A$ is not a counterexample to Theorem \ref{thmA}.
\end{lem}
\begin{proof}
The group $\type{B}_3(3)$ and $\PSU_4(3)$ are dealt with in Proposition \ref{prop:sporadics}, so our conditions on $n$ and $q$ mean that we may assume $S$ does not have an exceptional Schur multiplier.  Then let $\bG$ be a simple algebraic group over $\overline{\mathbb{F}}_q$ of simply connected type such that $G=\bG^F$ satisfies $G/\zent{G}\cong S$. Then by \cite[Theorem 21.14]{CE04}, every unipotent character of $G$ lies in the principal $2$-block $B_0(G)$ of $G$, and the non-unipotent characters in $B_0(G)$ are exactly those characters lying in Lusztig series $\mathcal{E}(G,s)$ with $s$ a $2$-element of $G^\ast$.  Since unipotent characters of $G$ are trivial on the center, we see that $k_A(B_0(S))\geq 6$ just by considering unipotent characters and also taking into consideration \cite[Theorem 2.5]{malle08},  except possibly in the case $\PSp_4(q)$, $\PSL_3^\epsilon(q)$, or $\PSL_4^\epsilon(q)$.  
In the case of $\PSp_4(q)$, the results of \cite{white90a} show that there are at least six characters of distinct degree in the principal block of $\Sp_4(q)$ that are trivial on the center, which forces $k_A(B_0(S))\geq 6$. 
 In the case $S=\PSL_4^\epsilon(q)$, $S$ has five unipotent characters, which are $\aut(S)$-invariant using \cite[Theorem 2.5]{malle08}, and hence the result is obtained by considering a character of $G$ in a series indexed by any $2$-element of $\Oh{q_0'}{G^\ast}=[G^\ast, G^\ast]$, which will be trivial on the center using \cite[Proposition 2.6(iii)]{SFT20}.
 Finally, in the case $S=\PSL_3^\epsilon(q)$, we may argue as in the last paragraph of Proposition \ref{prop:typeAsmall}, but taking  $|a|=4$ in place of $|a|=p$.
\end{proof}

\begin{lem}
Let $q$ be a power of an odd prime and let $p=2$.  Let $A$ be an almost simple group with simple socle $S=\PSL_2(q)$. Then $A$ is not a minimal counterexample to Theorem \ref{thmA}.
\end{lem}
\begin{proof}
In this case, $B_0(S)$ contains two unipotent characters.  Taking $t\in [\wt{G}^\ast, \wt{G}^\ast]$ to have eigenvalues $\{a,a^{-1}\}$ with $|a|=4$, the character $\chi_t$ restricts to the sum of two nonunipotent irreducible characters of $G$ trivial on the center.  We therefore have $k_A(B_0(S))\geq 3$, so we may assume $[A:S]$ is odd by applying Lemma \ref{lem:quot4} and Theorem \ref{thmB}. Then a Sylow $2$-subgroup of $A$ is Dihedral or Klein-four.  If $|P|=2^n\geq 8$, then $k(B_0(A))=2^{n-2}+3$ by \cite[Theorem 8.1]{sambalebook}, which is larger than $5$ for $n>3$.  If $|P|=8$, then $P$ is ${\sf D}_8$ and $A$ is not a counterexample.   
If $|P|=4$, then $(q^2-1)_2=8$, and every semisimple $2$-element in $[\wt{G}^\ast, \wt{G}^\ast]$ is $\wt{G}^\ast$-conjugate to $t$ above. This means that the only characters of $B_0(S)$ are the four discussed at the beginning of the proof, which are $A$-invariant. Further, we have $P={\sf C}_2\times {\sf C}_2$. Since $2\nmid [A:S]$, we know $A/S$ is cyclic generated by field automorphisms. We then see, using the construction in \cite{carterfong}, that a generating field automorphism centralizes the Sylow $2$-subgroup of $G$, modulo $\zent{G}$.  That is, a generator of $A/S$ centralizes $P$, contradicting the assumption from Theorem \ref{thmB} that $S\cent{A}{P}\neq A$.
\end{proof}

\noindent {\bf Acknowledgement.} First of all, we would like to thank Gabriel Navarro for bringing this topic to our attention. We also thank Benjamin Sambale for comments on a previous version. Finally, we are obliged to the anonymous referees for accurate corrections on our article.

\end{document}